\documentclass[12pt]{article}
\usepackage{amsmath,amsfonts,amssymb,amsthm, bbm}
\usepackage{epsf,epsfig,psfrag}

\DeclareMathOperator{\dist}{dist}

\DeclareMathOperator{\co}{co}
\DeclareMathOperator{\Proj}{Proj}

\begin{document}
\newtheorem{oberklasse}{OberKlasse}
\newtheorem{lemma}[oberklasse]{Lemma}
\newtheorem{proposition}[oberklasse]{Proposition}
\newtheorem{theorem}[oberklasse]{Theorem}
\newtheorem{remark}[oberklasse]{Remark}
\newtheorem{corollary}[oberklasse]{Corollary}
\newtheorem{definition}[oberklasse]{Definition}
\newtheorem{example}[oberklasse]{Example}
\newtheorem{observation}[oberklasse]{Observation}
\newcommand{\clconvhull}{\ensuremath{\overline{\co}}}
\newcommand{\R}{\ensuremath{\mathbbm{R}}}
\newcommand{\N}{\ensuremath{\mathbbm{N}}}
\newcommand{\Z}{\ensuremath{\mathbbm{Z}}}
\newcommand{\ClSets}{\ensuremath{\mathcal{A}}}
\newcommand{\CpSets}{\ensuremath{\mathcal{C}}}
\newcommand{\CoCpSets}{\ensuremath{\mathcal{CC}}}
\newcommand{\powerset}{\ensuremath{\mathcal{P}}}

\renewcommand{\phi}{\ensuremath{\varphi}}
\renewcommand{\epsilon}{\ensuremath{\varepsilon}}

\title{A numerical method for the solution of relaxed one-sided Lipschitz algebraic inclusions}
\author{Wolf-J\"urgen Beyn\footnote{Supported by CRC 701 
'Spectral Structures and Topological Methods in Mathematics', Bielefeld University.}\\ 
Fakult\"at f\"ur Mathematik, Universit\"at Bielefeld\\
Postfach 100131, D-33501 Bielefeld, Germany\\
Janosch Rieger\\
Institut für Mathematik, Universit\"at Frankfurt\\
Postfach 111932, D-60054 Frankfurt a.M., Germany}
\date{\today }
\maketitle

\begin{abstract}
An existing solvability result for relaxed one-sided Lipschitz algebraic inclusions
is substantially improved. This enhanced solvability result allows the design of a very robust
numerical method for the approximation of a solution of the algebraic inclusion. Sharp error estimates
for this method, illustrative analytic examples and a numerical example are provided.
\end{abstract}

\section{Introduction and notation}

The solution of nonlinear equations and inclusions is one of the fundamental problems in pure and
applied mathematics. A multitude of analytical concepts for the identification and localization of solutions
as well as numerical methods for their approximation have been developed that exploit characteristic features 
of particular types of mappings. 
In this paper, solutions of the algebraic inclusion 
\begin{equation} \label{basic}
\bar y\in F(x) 
\end{equation}
with given $\bar y\in\R^d$ are 
considered for the class of relaxed one-sided Lipschitz (ROSL, see below) multivalued mappings $F$ with negative one-sided 
Lipschitz constant. 
The relatively modern ROSL property was introduced and investigated in \cite{Donchev:96} and other works of the same author.
It generalizes the classical one-sided Lipschitz property and is a key criterion for the analysis of differential
inclusions and numerical approximations of their solution sets (see e.g.\ \cite{Donchev:Farkhi:98}),
so that algebraic inclusions of type \eqref{basic} with ROSL multifunction $F$ arise in a natural way.
Moreover, the ROSL property is intimately related to the notion
of metric regularity, which is discussed in \cite[Chapter 3]{Dontchev:Rockafellar:10}.

A solvability result for the class of multivalued mappings satisfying the ROSL property was proved in 
\cite[Corollary 3]{Beyn:Rieger:10}. It states that given an initial guess $x$, there exists a solution
$\bar x$ of \eqref{basic} in a closed ball centered at $x$ with radius depending
on the defect $\dist(\bar y,F(x))$.
A substantially improved version of this result is given in Theorem \ref{loesbarkeitssatz} below, 
which allows to localize a solution of \eqref{basic} in a smaller ball $B$ with $x\in\partial B$
and thus specifies not only a distance but also a direction in which a solution is to be found
(see Figure \ref{figure:loesbarkeit}). If the mapping $F$ is in addition Lipschitz continuous,
then the localization of the solution can once again be improved. 

This information can be used to design a very robust numerical algorithm
for the approximation of a solution of \eqref{basic} that uses the current state as initial guess for the 
improved solvability theorem and defines the next iterate as the center of the ball $B$. 
Proposition \ref{method} provides error estimates for this numerical scheme, and Example 
\ref{perp:example} shows that they are sharp for dimension $d>1$. The one-dimensional case 
is treated separately in Proposition \ref{method:1d}. Enhancements of the numerical method 
for $L$-Lipschitz multimaps $F$ are briefly analyzed in Propositions \ref{Lip:prop:1} and \ref{Lip:prop:2},
and a numerical example is provided.

\bigskip

Let $\R^d$ be equipped with the Euclidean norm $|\cdot|$ and the Euclidean inner product $\langle\cdot,\cdot\rangle$. 
A closed ball with radius $R\ge 0$ centered at $x\in\R^d$
will be denoted by $B_R(x)=B(x,R)$. The family of nonempty compact and convex subsets
of $\R^d$ is denoted by $\CoCpSets(\R^d)$, the one-sided and the symmetric Hausdorff-distances
of two sets $A,B\in\CoCpSets(\R^d)$ are defined by
\begin{align*}
\dist(A,B) &:= \sup_{a\in A}\inf_{b\in B}|a-b|,\\
\dist_H(A,B) &:= \max\{\dist(A,B),\dist(B,A)\},
\end{align*}
and the so-called norm of a set $A\in\CoCpSets(\R^d)$ is $\|A\|:=\max_{a\in A}|a|$.
The metric projection of a point $y\in\R^d$ to a set $A\in\CoCpSets(\R^d)$ is the unique point
$\Proj(y,A)\in A$ satisfying $|y-\Proj(y,A)|=\dist(y,A)$.

Consider a multivalued mapping $F:\R^d\rightarrow\CoCpSets(\R^d)$.
It is called upper semicontinuous (usc) at $x\in\R^d$ if
\[\dist(F(x'),F(x))\rightarrow 0\ \text{as}\ x'\rightarrow x,\]
usc if it is usc at every $x\in\R^d$,
and $L$-Lipschitz with $L\ge 0$ if 
\[\dist_H(F(x),F(x'))\le L|x-x'|\ \text{for all}\ x,x'\in\R^d.\]
The mapping is called relaxed one-sided Lipschitz with constant $l\in\R$ ($l$-ROSL)
if for any $x,x'\in\R^d$ and $y\in F(x)$, there exists some $y'\in F(x')$ such that
\[\langle y-y',x-x'\rangle \le l|x-x'|^2.\]

\section{Solvability of ROSL algebraic inclusions}

The following theorem is the core of this paper. It is a strongly improved version of the 
solvability theorem given in \cite[Corollary 3]{Beyn:Rieger:10}, and its assumptions on the mapping $F$
can still be weakened (see Remark \ref{weakened:assumptions}). Its statement is illustrated in 
Figure \ref{figure:loesbarkeit}.

\begin{theorem}
\label{loesbarkeitssatz}
Let $F:\R^d\rightarrow\CoCpSets(\R^d)$ be usc and ROSL with constant $l<0$, and let $\tilde x\in\R^d$ and 
$\bar y\in\R^d$ be given. Then there exists a solution 
\[\bar x\in S_F(\bar y) := \{x\in\R^d: \bar y\in F(x)\}\]
satisfying
\begin{equation} \label{new:condition:2}
|\bar x -(\tilde x+\frac{1}{2l}(\bar y-\Proj(\bar y,F(\tilde x))))| \le -\frac{1}{2l}\dist(\bar y,F(\tilde x)),
\end{equation}
and the set $S_F(\bar y)$ is closed.
If $F$ is in addition $L$-Lipschitz, then for any $\bar x\in S_F(\bar y)$,
\begin{equation} \label{exclusion}
|\bar x-\tilde x| \ge \frac{1}{L}\dist(\bar y,F(\tilde x)).
\end{equation}
\end{theorem}

\begin{figure}[h]
\begin{center}
\psfrag{x}{$\tilde x$}
\psfrag{y}{$\bar{y}$}
\psfrag{xc}{$\tilde x_c$}
\psfrag{F(x)}{$F(\tilde x)$}
\psfrag{Proj}{$\Proj(\bar y,F(\tilde x))$}
\includegraphics[scale=0.5]{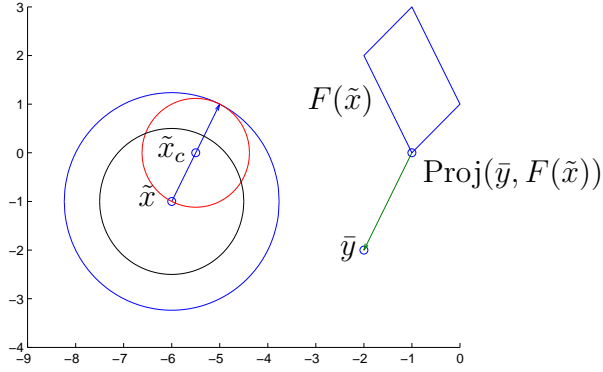}
\end{center}
\caption{Schematic illustration of Theorem \ref{loesbarkeitssatz}. The solvability theorem
given in \cite[Corollary 3]{Beyn:Rieger:10} only guarantees the existence of a solution 
$\bar x$ of $\bar y\in F(x)$ in the (blue) ball of radius
$-\frac{1}{l}\dist(\bar y,F(\tilde x))$ centered at $\tilde x$. Theorem \ref{loesbarkeitssatz} guarantees such
a solution in the (red) ball with radius $-\frac{1}{2l}\dist(\bar y,F(\tilde x))$ centered at 
$\tilde x_c=x+\frac{1}{2l}(\bar y-\Proj(\bar y,F(\tilde x)))$,
and if $F$ is $L$-Lipschitz, it states that no solution is contained in the (black) ball of
radius $\frac{1}{L}\dist(\bar y,F(\tilde x))$ centered at $\tilde x$.
\label{figure:loesbarkeit}}
\end{figure}

\begin{lemma}
\label{zero:solution:inclusion}
Let $F:\R^d\rightarrow\CoCpSets(\R^d)$ be usc and ROSL with constant $l<0$. Then the inclusion $0\in F(x)$ 
has a solution $\bar x$ with 
\begin{equation}
|\bar x| \leq -\frac{1}{l}\dist(0,F(0))
\end{equation}
that satisfies the property
\begin{equation} \label{new:condition}
\langle -\Proj(0,F(0)),\bar x\rangle \le l|\bar x|^2.
\end{equation}
\end{lemma}

\begin{proof}
Let $y_0:=\Proj(0,F(0))$ be the element with minimal norm.
By the ROSL property of $F$, the mapping $\Psi:\R^d\rightarrow\CoCpSets(\R^d)$ given by
\[\Psi(x) := F(x) \cap \{y\in\R^d: \langle y-y_0,x\rangle \le l|x|^2\}\]
has nonempty images. By \cite[Theorem 1.1.1]{Aubin:Cellina:84}, it is usc.
Define the usc mapping $G:\R^d\rightarrow\CoCpSets(\R^d)$ by
\begin{equation*}
G(x) := x+\alpha \Psi(x)
\end{equation*}
with some $\alpha>0$. Take $y\in\Psi(x)$ and set $z:=x+\alpha y$. Then
\begin{eqnarray*}
|z|^2 &=& |x|^2 + 2\alpha\langle y,x\rangle + \alpha^2|y|^2\\
&=& |x|^2 + 2\alpha\langle y-y_0,x\rangle + 2\alpha\langle y_0,x\rangle + \alpha^2|y|^2\\
&\leq& |x|^2+2\alpha l|x|^2 + 2\alpha|x|\dist(0,F(0))  + \alpha^2|y|^2.
\end{eqnarray*}
Thus, if $R>-\frac{1}{l}\dist(0,F(0))$, $|x|\leq R$, and $\alpha$ is so small that $1+2\alpha l\geq 0$, then
\begin{eqnarray}
|z|^2 &\leq& R^2 + 2\alpha(lR+\dist(0,F(0)))R + \alpha^2|y|^2\nonumber\\
&<& R^2 + \alpha^2|y|^2 \label{fast:fertig}. 
\end{eqnarray}
As $F$ is usc, 
\begin{equation*}
M_R:=\sup_{x\in B_R(0)}\|F(x)\| < \infty,
\end{equation*}
and there exists an $\alpha>0$ such that $|z|^2\leq R^2$ follows from \eqref{fast:fertig}.
This means that for this fixed $\alpha$, 
\begin{equation*}
H(x):=G(x)\cap B_R(0) \neq \emptyset \text{ for all } x\in B_R(0),
\end{equation*}
and $H(\cdot)$ is also usc. By the Kakutani Theorem (see \cite[Theorem 3.2.3]{Aubin:Frankowska:90}),
$H$ and thus also $G$ have a fixed point $x_R$ in $B_R(0)$, which implies that
$0\in \Psi(x_R)$. 

In particular, we find elements $x_n\in B(0,-\frac{1}{l}\dist(0,F(0))+1/n)$ for all $n\in\N$ 
such that $0\in \Psi(x_n)$. As $B(0,-\frac{1}{l}\dist(0,F(0))+1)$ is compact, there exists
a convergent subsequence of $\{x_n\}_{n\in\N}$ with limit 
\begin{equation*}
\bar x\in B(0,-\frac{1}{l}\dist(0,F(0))).
\end{equation*}
Since $\Psi$ is usc, 
\begin{equation*}
0\in \Psi(\bar x) \subset F(\bar x).
\end{equation*}
Property \eqref{new:condition} follows from the construction of $\Psi$.
\end{proof}

\begin{proof}[Proof of Theorem \ref{loesbarkeitssatz}]
Consider the set-valued mapping
\begin{equation*}
G(z):=F(z+\tilde x) -\bar y,
\end{equation*}
which is ROSL with constant $l$. By the above theorem, there exists some
\begin{equation*}
\bar z\in B(0, -\frac{1}{l}\dist(0,G(0))) = B(0,-\frac{1}{l}\dist(\bar y, F(\tilde x)))
\end{equation*}
such that $0\in G(\bar z)$ or $\bar y\in F(\bar x)$, where $\bar x=\tilde x+\bar z$.
Property \eqref{new:condition},
\[\langle -\Proj(0,G(0)),\bar z\rangle = \langle -\Proj(0,F(\tilde x)-\bar y),\bar x-\tilde x\rangle 
= \langle \bar y-\Proj(\bar y,F(\tilde x)),\bar x-\tilde x \rangle,\]
and 
\[l|\bar z|^2 = l|\bar x-\tilde x|^2.\]
imply that
\[\langle \bar y-\Proj(\bar y,F(\tilde x)),\bar x-\tilde x\rangle \le l|\bar x-\tilde x|^2,\]
which is equivalent with \eqref{new:condition:2}. 

The fact that $S_F(\bar y)$ is closed follows directly from the usc property of $F$.

If $F$ is in addition $L$-Lipschitz and $\bar x\in S_F(\bar y)$, then 
\[\dist(\bar y,F(\tilde x)) \le \dist(F(\bar x),F(\tilde x)) \le L|\bar x-\tilde x|\]
implies
\[|\bar x-\tilde x| \ge \frac{1}{L}\dist(\bar y,F(\tilde x)).\]
\end{proof}

\begin{remark} \label{weakened:assumptions}
The assumptions of Theorem \ref{loesbarkeitssatz} can be weakened. 
In particular, the set-valued mapping $F$ may be defined only on $B:=B(\tilde x,-\frac{1}{l}\dist(\bar y,F(\tilde x)))$.
\begin{itemize}
\item [a)] In order to obtain
the existence of a solution and estimate \eqref{new:condition:2}, it is sufficient to require that
$F:B\rightarrow\CoCpSets(\R^d)$ is usc and that 
\begin{equation} \label{weaker}
\forall\, x\in B\ \exists\, y\in F(x):\ \langle y-\Proj(\bar y, F(\tilde x)),x-\tilde x\rangle \le l|x-\tilde x|^2. 
\end{equation}
The mapping $F$ can then be extended as in \cite[proof of Theorem 2]{Beyn:Rieger:10a} to a 
set-valued function $\tilde F:\R^d\rightarrow\CoCpSets(\R^d)$ that coincides with $F$ on
$B$, is usc, and satisfies property \eqref{weaker} 
for all $x\in\R^d$. The proof of Theorem \ref{loesbarkeitssatz} can be applied to the 
mapping $\tilde F$ without changes.
\item [b)] To show estimate \eqref{exclusion}, it is enough for 
$F:B\rightarrow\CoCpSets(\R^d)$
to be $L$-Lipschitz relative to $\tilde x$ in the sense that
\[\dist_H(F(x),F(\tilde x)) \le L|x-\tilde x|\ \text{for all}\ x\in B.\]
It follows directly that for any $\bar x\in S_F(\bar y)\cap B$ (and hence for all $\bar x\in S_F(\bar y)$),
\[|\bar x-\tilde x| \ge \frac{1}{L}\dist(\bar y,F(\tilde x)).\]
\end{itemize}
\end{remark}

\begin{remark}
It is unclear if additional assumptions are needed to guarantee the connectedness of $S_F(\bar y)$. 
This question is linked with the parametrization problem
for ROSL multifunctions (see Lemma 12 in \cite{Beyn:Rieger:10}).
\end{remark}

\section{A numerical solver for ROSL algebraic inclusions}

Throughout this section, the mapping $F:\R^d\rightarrow\CoCpSets(\R^d)$ will be assumed to be
$l$-ROSL and $L$-Lipschitz.
A numerical method for finding a solution $\bar x$ of the inclusion $\bar y\in F(x)$ can be deduced 
directly from Theorem \ref{loesbarkeitssatz} by defining the next iterate of the scheme as the center 
of the ball specified by \eqref{new:condition:2}.

\begin{proposition} \label{method}
Let $L<-2l$, and let $x_0\in\R^d$ and $\bar y\in\R^d$ be given. Then the sequence $\{x_n\}_{n\in\N}$ defined by
\begin{equation} \label{iteration}
x_{n+1} := \Phi(x_n) := x_n + \frac{1}{2l}(\bar y-\Proj(\bar y,F(x_n)))
\end{equation}
converges to a solution $\bar x$ of the inclusion $\bar y\in F(x)$ and satisfies the estimates
\begin{equation} \label{distance:to:set}
\dist(x_n,S_F(\bar y)) \le \frac{L^{n-1}}{|2l|^n}\dist(\bar y,F(x_0))
\end{equation}
and
\begin{equation} \label{distance:to:point}
|x_n-\bar{x}| \le -\frac{1}{2l}\frac{\frac{L^n}{|2l|^n}}{1-\frac{L}{|2l|}}\dist(\bar y,F(x_0))
\end{equation}
for $n\ge 1$.
\end{proposition}

%\begin{remark}
%The algorithm \eqref{iteration} approximates $S_F(\bar y)$ in the sense that 
%$x_n\rightarrow\bar x\in S_F(\bar y)$ as $n\rightarrow\infty$.
%It is in principle possible to define a set-valued modification of algorithm \eqref{iteration} that approximates
%$S_F(\bar y)$ in the sense that its iterates $X_n$ are subsets of $\R^d$ and $\dist_H(X_n,S_F(\bar y))\rightarrow 0$
%as $n\rightarrow\infty$. This is, however, inefficient, because $S_F(\bar y)$ is connected
%(see the discussion in \cite{Beyn:Rieger:10} and \cite{Beyn:Rieger:10a}). If the whole solution
%set $S_F(\bar y)$ is to be approximated, it is therefore much cheaper to approximate one $\bar x\in S_F(\bar y)$
%and then check the defect $\dist(\bar y,F(x))$ of elements $x$ in a neighborhood of $\bar x$.
%\end{remark}

\begin{proof}
Set $v_n:=\bar y-\Proj(\bar y,F(x_n))$ for $n\in\N$. Then \eqref{new:condition:2} implies
that there exists some $\bar x_n\in S_F(\bar y)$ such that
\begin{align} \label{vn:estimate}
\dist(x_{n+1},S_F(\bar y)) \le |\bar x_n-(x_n+\frac{1}{2l}v_n)| \le -\frac{1}{2l}|v_n|.
\end{align}
Now 
\begin{align*}
|v_{n+1}| &= \dist(\bar y,F(x_{n+1})) \le \dist(\bar y,F(\bar x_n)) + \dist(F(\bar x_n),F(x_n + \frac{1}{2l}v_n))\\
&\le L|\bar x_n-(x_n+\frac{1}{2l}v_n)| \le -\frac{L}{2l}|v_n|
\end{align*}
by \eqref{vn:estimate} for $n\in\N$, so that 
\[|v_n| \le \frac{L^n}{|2l|^n}|v_0|,\]
and again by \eqref{vn:estimate}, we have
\begin{equation} \label{distance:to:set:proof}
\dist(x_n,S_F(y)) \le -\frac{1}{2l}|v_{n-1}| \le \frac{L^{n-1}}{|2l|^n}|v_0|
\end{equation}
for $n\ge 1$, which shows \eqref{distance:to:set}.
Since
\begin{equation} \label{subsequent}
|x_{n+1}-x_n| \le -\frac{1}{2l}|v_n| \le -\frac{1}{2l}\frac{L^n}{|2l|^n}|v_0|
\end{equation}
for all $n\in\N$, the sequence $\{x_n\}_{n\in\N}$ is Cauchy and converges to some $\bar{x}\in\R^d$.
As $S_F(\bar y)$ is closed, estimate \eqref{distance:to:set:proof} shows that $\bar{x}\in S_F(\bar y)$.
Finally, for all $n,N\in\N$ with $N>n$, it follows from \eqref{subsequent} that
\begin{align*}
|x_N-x_n| &\le \sum_{j=n}^{N-1}|x_{j+1}-x_j| \le -\frac{1}{2l}|v_0|\sum_{j=n}^{N-1}\frac{L^j}{|2l|^j}\\
&= -\frac{1}{2l}|v_0|\frac{L^n}{|2l|^n}\sum_{j=0}^{N-n-1}\frac{L^j}{|2l|^j} 
= -\frac{1}{2l}|v_0|\frac{L^n}{|2l|^n}\frac{1-\frac{L^{N-n}}{|2l|^{N-n}}}{1-\frac{L}{|2l|}}\\
&\le -\frac{1}{2l}|v_0|\frac{\frac{L^n}{|2l|^n}}{1-\frac{L}{|2l|}}.
\end{align*}
Passing to the limit as $N\rightarrow\infty$ yields \eqref{distance:to:point}.
\end{proof}

\begin{remark}
By Theorem \ref{loesbarkeitssatz}, any numerical iteration $\{x_n\}_{n\in\N}$ will
converge to $S_F(y)$ provided that the sequence of defects $v_n$ converges to zero.
A simple modification of the proof of Proposition \ref{method} shows that the defect 
at any point in the interval $x_n+(\frac{1}{l},0)v_n$ is smaller
than at $x_n$ and that for $r\in[0,-\frac{1}{2l})$
and $|x_{n+1}-(x_n+\frac{1}{2l}v_n)|\le r$ for all $n\in\N$, the algorithm still converges linearly
with reduced speed. This means that even if $l$ is unknown, it is possible to find a next iterate with 
smaller defect according to simple trust region strategies.
\end{remark}

The following example shows that Proposition \ref{method} is sharp
(apart from statement \eqref{distance:to:point}).
\begin{example} \label{perp:example}
Let $l<0$ and $L\ge-l$, and set $F(x):=lx+\alpha x^\perp$, where $\alpha:=\sqrt{L^2-l^2}$ and $x^\perp:=(x^{(2)},-x^{(1)})$
is the image of $x$ under the rotation with angle $-\pi/2$ around the origin. The single-valued mapping $F$ is $l$-OSL and $L$-Lipschitz.
If the numerical method \eqref{iteration} is applied to the problem $0=F(x)$, we have
\[\Phi(x)=x-\frac{1}{2l}F(x) = \frac12\begin{pmatrix}1 & -\alpha/l\\ \alpha/l & 1\end{pmatrix}x.\]
The eigenvalues of the above matrix are $\lambda_{1/2}=\frac12\pm\frac{\alpha}{2l}i$, i.e.\ the
iteration converges if and only if $L<-2l$. Moreover,
\[\|\frac12\begin{pmatrix}1 & -\alpha/l\\ \alpha/l & 1\end{pmatrix}\|_2 = -\frac{L}{2l},\]
so that the iteration converges with rate $-\frac{L}{2l}$ whenever $L<-2l$. 
In fact, it can be shown easily by using rotational symmetry of 
$F$ that estimate \eqref{distance:to:set} is sharp for every initial state $x_0\in\R^2$.
\end{example}

The following example shows that the condition $L<-2l$ is not sharp for convergence of the method \eqref{iteration}
in $d=1$.
\begin{example} \label{1d:example}
Consider the function $F:\R\rightarrow\R$ given by
\[F(x)=\left\{\begin{array}{ll}-L+l(x-1),& 1\le x\\ -Lx,& -1\le x\le 1\\ +L+l(x+1),& x\le -1\end{array}\right.\]
with $l<0$ and $L\ge-l$. Clearly, $F$ is $l$-OSL and $L$-Lipschitz. 
Let $x_n\in[-1,1]$ be a state of the root finding method that is supposed to solve $0=F(x)$.
Then 
\[x_{n+1} = x_n-\frac{F(x_n)}{2l} = x_n+\frac{Lx_n}{2l} = (1+\frac{L}{2l})x_n,\]
so that $|x_{n+1}|<|x_n|$ if and only if $L<-4l$. Figure \ref{1d:figure} illustrates the global behavior
of the function $F$ and the numerical method $\Phi$ for characteristic ratios $-L/l$.
\end{example}

\begin{figure}
\begin{center}
\includegraphics[scale=0.6]{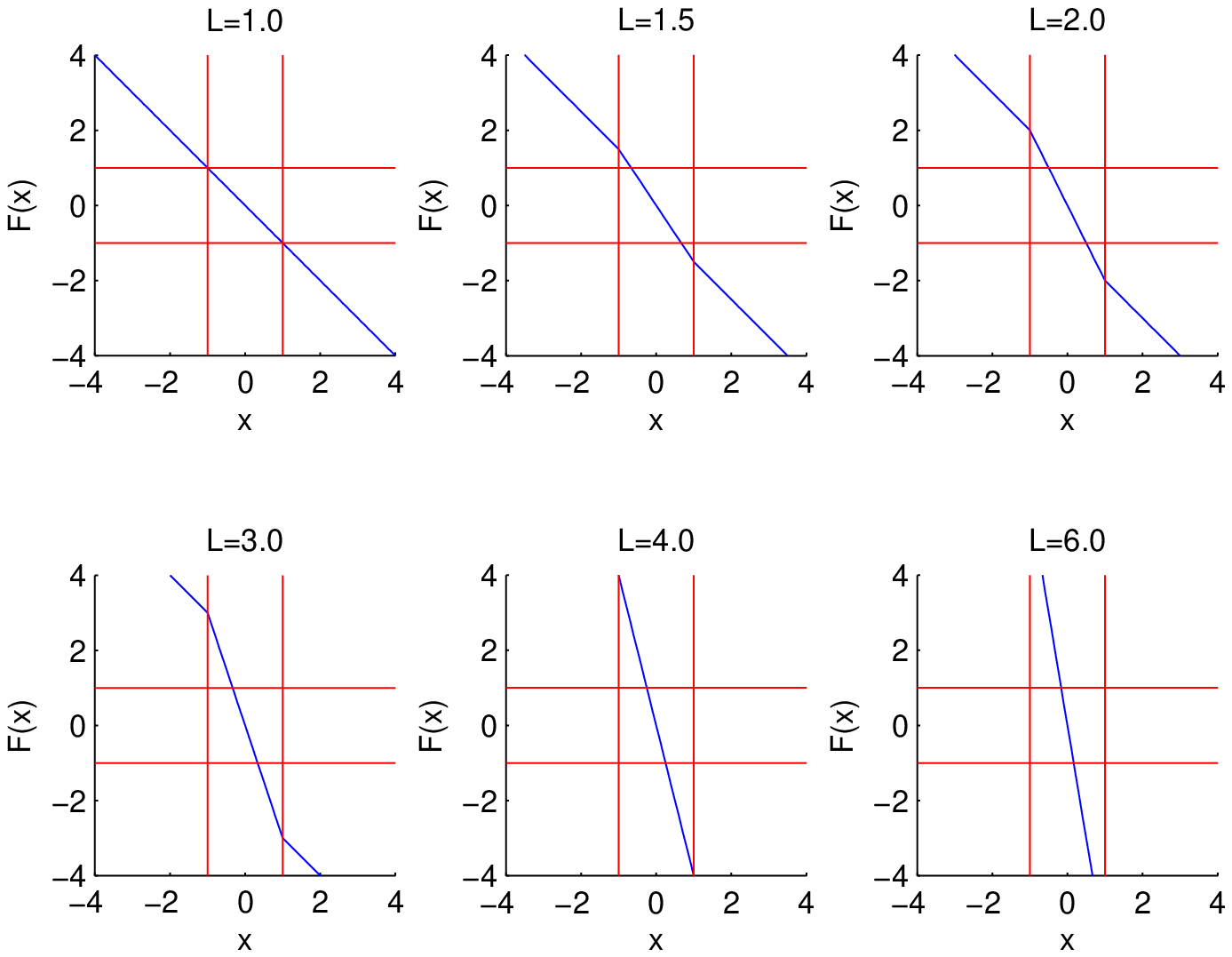}\\
%\end{center}
%\begin{center}
\includegraphics[scale=0.6]{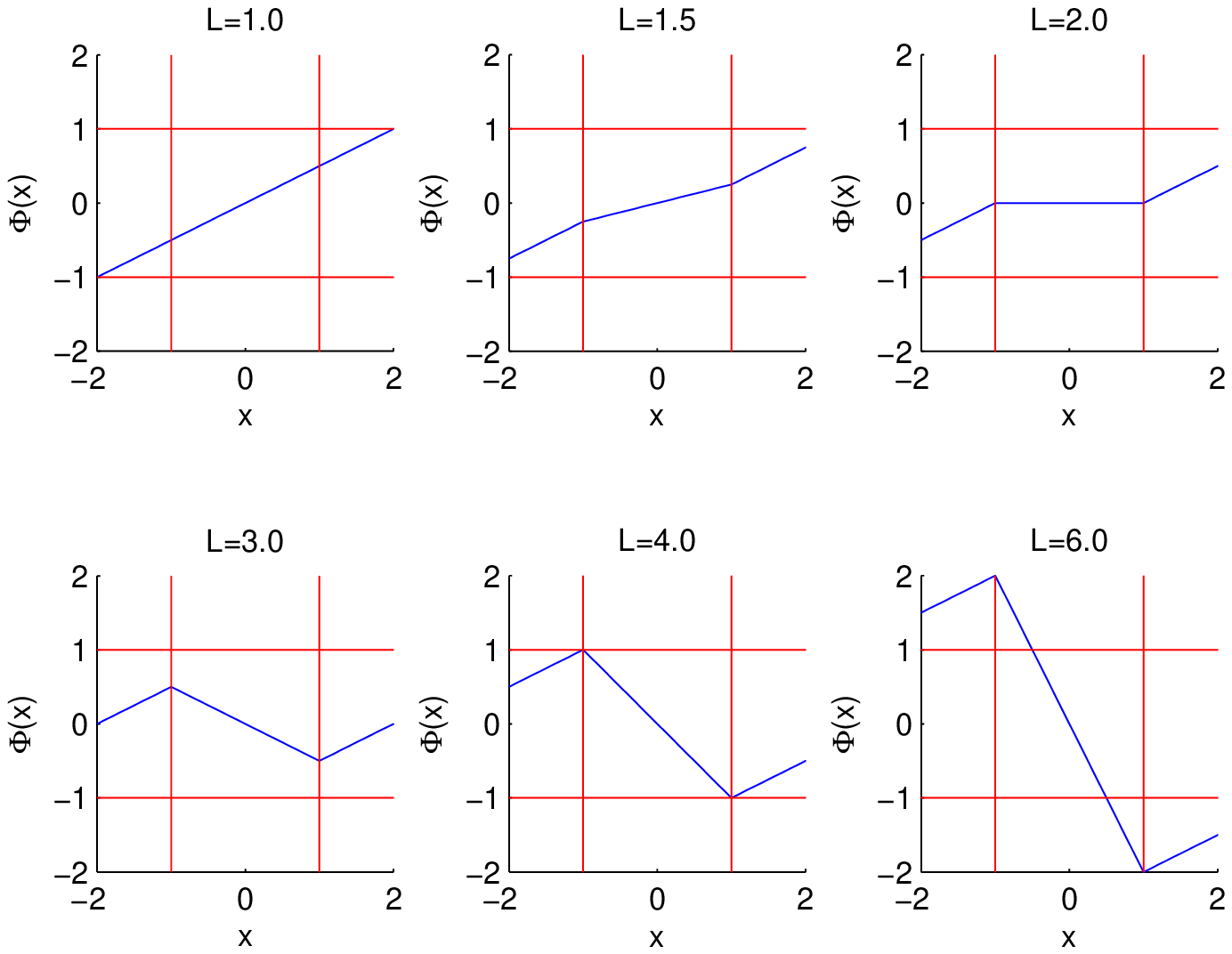}
\end{center}
\caption{Behavior of the function $F$ from Example \ref{1d:example} and the corresponding numerical method $\Phi$
for $l=-1$ and characteristic values of $L$. The red lines limit the central interval $[-1,1]$ 
in space and image. The value $L=-4l$ is the critical threshold. \label{1d:figure}}
\end{figure}

The gap between the condition $L<-2l$ required for convergence in Proposition \ref{method} and the condition $L<-4l$
observed in Example \ref{1d:example} is due to the fact that for multifunctions $F:\R\rightarrow\CoCpSets(\R)$,
the ROSL property is much stronger than in $\R^d$ with $d>1$. In this particular context, it is possible to 
derive estimates for some of the defects (see Case 1a in the following proof) that only depend on the one-sided 
Lipschitz constant $l$ and not on the Lipschitz constant $L$.

\begin{proposition} \label{method:1d}
Let $F:\R\rightarrow\CoCpSets(\R)$ be $l$-ROSL and $L$-Lipschitz with $l<0$ and $L<-4l$, and let $x_0\in\R$ and 
$\bar y\in\R$ be given. Then the sequence $\{x_n\}_{n\in\N}$ defined by
\[x_{n+1} := x_n + \frac{1}{2l}(\bar y-\Proj(\bar y,F(x_n)))\]
converges to a solution $\bar x$ of the inclusion $\bar y\in F(x)$ and satisfies the estimates
\begin{equation} \label{distance:to:set:1d}
\dist(x_n,S_F(\bar y)) \le -\frac{1}{2l}\kappa^{n-1}\dist(\bar y,F(x_0))
\end{equation}
and
\begin{equation} \label{distance:to:point:1d}
|x_n-\bar{x}| \le -\frac{1}{2l}\frac{\kappa^n}{1-\kappa}\dist(\bar y,F(x_0))
\end{equation}
for $n\ge 1$, where $\kappa:=\max\{\frac12,|1+\frac{L}{2l}|\}$.
\end{proposition}

%In view of Example \ref{1d:example}, Proposition \ref{method:1d} (without statement \eqref{distance:to:point:1d}) is sharp.

\begin{proof}
Let $-2l\le L <-4l$ and set $v_n:=\bar y-\Proj(\bar y,F(x_n))$ for $n\in\N$. Without loss of generality, $\bar y\notin F(x_n)$ and
$\bar y\notin F(x_{n+1})$, because otherwise the sequences $\{v_n\}$ and $\{x_n\}$ become constant and 
all estimates are trivially satisfied. As $F(x_n)$ is an interval, there are only two cases.

\emph{Case 1:} $\bar y>y$ for all $y\in F(x_n)$.\\
In particular, $v_n>0$. If $\bar x\in S_F(\bar y)$, then the ROSL property yields some $y\in F(x_n)$
such that 
\[(\bar y-y)(\bar x-x_n) \le l|\bar x-x_n|^2,\]
which implies $\bar x\le x_n$. By Theorem \ref{loesbarkeitssatz}, 
\[S_n:=S_F(\bar y)\cap[x_n+\frac{1}{l}v_n,x_n-\frac{1}{L}v_n]\neq\emptyset.\]
Let $\bar x_n:=\max S_n$. Without loss of generality, $x_n\neq \bar x_n\neq x_{n+1}$, because otherwise the
sequences $\{v_n\}$ and $\{x_n\}$ become constant. There are two subcases.

\emph{Subcase 1a:} $\bar x_n\in[x_n+\frac{1}{l}v_n,x_n+\frac{1}{2l}v_n)$.\\
Assume that there exists some $y^*\in F(x_{n+1})$ with $\bar y<y^*$. Since $y<\bar y$ for all $y\in F(x_n)$,
there exists some $x^*\in(x_{n+1},x_n)$ with $\bar y\in F(x^*)$ by the set-valued intermediate value theorem
(see Appendix). But then $x^*\in S_F(\bar y)$, which contradicts the maximality of $\bar x_n$.
Therefore,
\begin{equation} \label{Fnplus1:on:correct:side}
\bar y>y\ \text{for all}\ y\in F(x_{n+1}), 
\end{equation}
and
\[\Proj(\bar y,F(x_n))=\max F(x_n), \Proj(\bar y,F(x_{n+1}))=\max F(x_{n+1}).\]
It is easy to see that if $F$ is $l$-ROSL, then the single-valued function $\max F$ is $l$-OSL, and hence
\begin{align*}
&\frac{1}{2l}v_n[\Proj(\bar y,F(x_{n+1}))-\Proj(\bar y,F(x_n))]\\
&=[\Proj(\bar y,F(x_{n+1}))-\Proj(\bar y,F(x_n))](x_{n+1}-x_n)\\
&=(\max F(x_{n+1})-\max F(x_n))\cdot(x_{n+1}-x_n)\\
&\le l|x_{n+1}-x_n|^2 \le \frac{1}{4l}v_n^2,
\end{align*}
which implies
\[\Proj(\bar y,F(x_{n+1}))-\Proj(\bar y,F(x_n))\ge\frac12 v_n\]
and thus 
\[\bar y-\Proj(\bar y,F(x_n)) - \frac12 v_n \ge \bar y-\Proj(\bar y,F(x_{n+1}))\]
and 
\[\frac12 v_{n}  \ge v_{n+1}.\]
Since $v_{n+1}>0$ by inequality \eqref{Fnplus1:on:correct:side},
\[|v_{n+1}| \le \frac12|v_n|.\]

\emph{Subcase 1b:} $\bar x_n\in(x_n+\frac{1}{2l}v_n,x_n-\frac{1}{L}v_n]$.\\
In this case, 
\begin{align*}
|v_{n+1}| &= \dist(\bar y,F(x_{n+1})) \le \dist(F(\bar x_n),F(x_{n+1})) \le L|\bar x_n-x_{n+1}|\\
&\le L|(x_n-\frac{1}{L}v_n) - (x_n+\frac{1}{2l}v_n)| \le L|\frac{1}{2l}+\frac{1}{L}|\cdot|v_n| = |1+\frac{L}{2l}|\cdot|v_n|.
\end{align*}

\emph{Case 2:} $\bar y<y$ for all $y\in F(x_n)$.\\
All arguments and estimates are symmetric to those in Case 1.

\bigskip

\noindent Summarizing Cases 1 and 2,
\[|v_{n+1}| \le \max\{\frac12,|1+\frac{L}{2l}|\}|v_n| =: \kappa|v_n|,\]
so that by induction,
\[|v_n| \le \kappa^n|v_0|.\]
By estimate \eqref{new:condition:2}, we have
\begin{equation} \label{distance:to:set:proof:1d}
\dist(x_n,S_F(\bar y)) \le -\frac{1}{2l}|v_{n-1}| \le -\frac{1}{2l}\kappa^{n-1}|v_0|
\end{equation}
for $n\ge 1$, which shows \eqref{distance:to:set:1d}.
Since
\begin{equation} \label{subsequent:1d}
|x_{n+1}-x_n| \le -\frac{1}{2l}|v_n| \le -\frac{1}{2l}\kappa^n|v_0|
\end{equation}
for all $n\in\N$, the sequence $\{x_n\}_{n\in\N}$ is Cauchy and converges to some $\bar{x}\in\R$.
As $S_F(y)$ is closed, estimate \eqref{distance:to:set:proof:1d} shows that $\bar{x}\in S_F(\bar y)$.
Finally, for all $N,n\in\N$ with $N>n$, it follows from \eqref{subsequent:1d} that
\begin{align*}
|x_N-x_n| &\le \sum_{j=n}^{N-1}|x_{j+1}-x_j| \le -\frac{1}{2l}|v_0|\sum_{j=n}^{N-1}\kappa^j\\
&= -\frac{1}{2l}|v_0|\kappa^n\sum_{j=0}^{N-n-1}\kappa^j 
= -\frac{1}{2l}|v_0|\kappa^n\frac{1-\kappa^{N-n}}{1-\kappa}\\
&\le -\frac{1}{2l}|v_0|\frac{\kappa^n}{1-\kappa}.
\end{align*}
Passing to the limit as $N\rightarrow\infty$ yields \eqref{distance:to:point:1d}.

\bigskip

If $L<-2l$, then Cases 1b and 2b cannot occur, so that all estimates hold with the optimal 
rate $\kappa=\frac12$.
\end{proof}

If the Lipschitz constant $L$ of the mapping $F$ is known explicitly, the numerical method
\eqref{iteration} can be refined using estimate \eqref{exclusion} from Theorem \ref{loesbarkeitssatz}.
The proofs will only be sketched, because they coincide in large parts with those of the above 
propositions.

\begin{proposition} \label{Lip:prop:1}
If $d>1$ and $L\le-\sqrt{2}l$, then the iteration
\[x_{n+1} := x_n+\frac{l}{L^2}(\bar y-\Proj(\bar y,F(x_n)))\]
converges to a solution $\bar x\in S_F(\bar y)$ and satisfies
\[\dist(x_n,S_F(\bar y)) \le -\frac{1}{2l}\kappa^{n-1}\dist(\bar y,F(x_0))\]
and
\[|x_n-\bar x| \le -\frac{l}{L^2}\frac{\kappa^n}{1-\kappa}\dist(\bar y,F(x_0)),\]
where $\kappa:=\frac{\sqrt{L^2-l^2}}{L}$.
\end{proposition}
\begin{proof}[Sketch of proof.]
Define $S_n:=B(x_n+\frac{1}{2l}v_n,-\frac{1}{2l}|v_n|)\setminus B(x_n,\frac{1}{L}|v_n|)$.
By Theorem \ref{loesbarkeitssatz}, there exists some $\bar x_n\in S_F(\bar y)\cap S_n$.
By simple geometric arguments, 
\[|\bar x_n-x_{n+1}| \le \dist(S_n,x_{n+1}) \le \frac{\sqrt{L^2-l^2}}{L^2}|v_n|,\]
so that
\begin{align*}
|v_{n+1}| &= \dist(\bar y,F(x_{n+1})) \le \dist(\bar y,F(\bar x_n)) + \dist(F(\bar x_n),F(x_{n+1}))\\
 &\le L|\bar x_n-x_{n+1}| \le \frac{\sqrt{L^2-l^2}}{L}|v_n| =: \kappa|v_n|.
\end{align*}
\end{proof}

The case $d=1$ allows more effective estimates.
\begin{proposition} \label{Lip:prop:2}
If $d=1$ and $L\le-2l$, then the iteration
\[x_{n+1} := x_n+\frac12(\frac{1}{l}-\frac{1}{L})(\bar y-\Proj(\bar y,F(x_n)))\]
converges to a solution $\bar x\in S_F(\bar y)$ and satisfies
\[\dist(x_n,S_F(\bar y)) \le -\frac{1}{2l}\kappa^{n-1}\dist(\bar y,F(x_0))\]
and
\[|x_n-\bar x| \le \frac12(\frac{1}{L}-\frac{1}{l})\frac{\kappa^n}{1-\kappa}\dist(\bar y,F(x_0))\]
for $n\ge 1$, where $\kappa:=\frac12(1-\frac{L}{l})$.
\end{proposition}
\begin{proof}[Sketch of proof.]
By Theorem \ref{loesbarkeitssatz}, there exists some $\bar x_n\in S_F(\bar y)\cap S_n$, where
\[S_n:=[x_n+\frac{1}{l}v_n,x_n]\setminus[x_n-\frac{1}{L}v_n,x_n+\frac{1}{L}v_n] 
= [x_n+\frac{1}{l}v_n,x_n-\frac{1}{L}v_n).\]
Therefore, 
\begin{align*}
|v_{n+1}| &= \dist(\bar y,F(x_{n+1})) \le \dist(\bar y,F(\bar x_n)) + \dist(F(\bar x_n),F(x_{n+1}))\\
 &\le L|\bar x_n-x_{n+1}| \le \frac{L}{2}|\frac{1}{L}-\frac{1}{l}|\cdot|v_n| =: \kappa|v_n|.
\end{align*}
\end{proof}

\begin{figure}[h]
\begin{center}
\psfrag{SF0}{$S_F(0)$}
\includegraphics[scale=0.35]{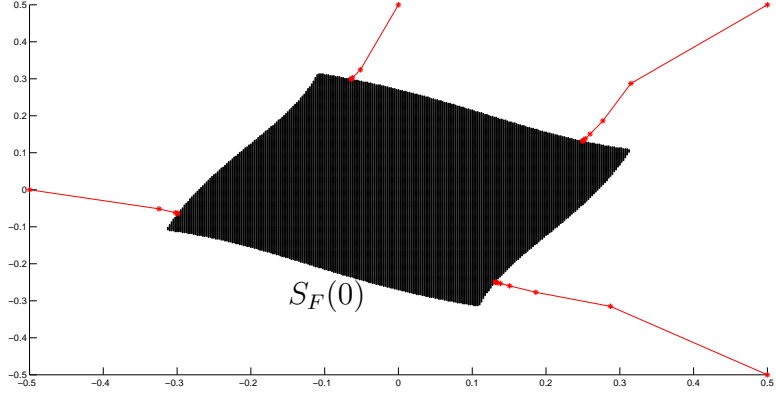}
\end{center}
\caption{Solution set $S_F(0)$ of inclusion \eqref{example:inclusion} and some typical trajectories
of the numerical scheme \eqref{iteration}.\label{numerical:result}}
\end{figure}

The following numerical example illustrates that the algorithm \eqref{iteration}
indeed approximates an element of the solution set $S_F(\bar y)$ successfully for any 
given initial value.

\begin{example}
Consider the multivalued mapping $F:\R^2\rightarrow\CoCpSets(\R^2)$ given by
\begin{equation} \label{example:inclusion}
F(x)=-3x+A(x)Q,
\end{equation}
where 
\[A(x)=\begin{pmatrix}\cos(|x|) & -\sin(|x|)\\ \sin(|x|) & \cos(|x|)\end{pmatrix}\quad 
\text{and}\quad Q=\overline{\text{co}}\{(1,0),(0,-1),(-1,0),(0,1)\}\]
are a rotation matrix with angle depending on the norm of $x$ and a square centered at the origin.
It is easy to check that $F$ is $(-2)$-ROSL and $3$-Lipschitz, so that the statements of
Proposition \ref{method} hold. The solution set $S_F(0)$ and typical trajectories of the numerical 
method \eqref{iteration} applied to the problem $0\in F(x)$ are depicted in Figure \ref{numerical:result}.
\end{example}

\section*{Appendix}

The proof of the following proposition does not differ much from that of the classical
intermediate value theorem and is therefore omitted.

\begin{proposition}
Let $a,b\in\R$ with $a<b$, and let $F:[a,b]\rightarrow\CoCpSets(\R)$ be an usc mapping such that there
exists some $f_a\in F(a)$ and $f_b\in F(b)$ with $f_a<0$ and $f_b>0$. Then there exists some $x^*\in(a,b)$
such that $0\in F(x^*)$.
\end{proposition}

\bibliographystyle{plain}
\bibliography{ROSLinclusion_arxiv}

\end{document}